\documentclass{amsart}
\usepackage{amsmath}
\usepackage{amssymb}
\usepackage{amsthm}

\usepackage{enumerate}
\usepackage{graphicx}
\usepackage{float}

\newtheorem{theorem}{Theorem}
\newtheorem{rem}{Remark}

\newtheorem{prop}{Proposition}
\newtheorem{lemma}{Lemma}

\usepackage{caption, float, graphicx, bxeepic, hyperref}

\def\arraystretch{0.9}
\usepackage[usenames, dvipsnames]{color}

\begin{document}
\title[Symmetric designs in terms of the geometry of binary simplex codes]{Symmetric $(15,8,4)$-designs in terms of the geometry of binary simplex codes of dimension $4$}
\author{Mark Pankov, Krzysztof Petelczyc, Mariusz \.Zynel}
\address{Mark Pankov: Faculty of Mathematics and Computer Science, 
University of Warmia and Mazury, S{\l}oneczna 54, 10-710 Olsztyn, Poland}
\email{pankov@matman.uwm.edu.pl}
\address{Krzysztof Petelczyc, Mariusz \.Zynel: Faculty of Mathematics, University of Bia{\l}ystok, Cio{\l}kowskiego 1M, 15-245 Bia{\l}ystok, Poland}
\email{kryzpet@math.uwb.edu.pl, mariusz@math.uwb.edu.pl}
\maketitle

\begin{abstract}
Let $n=2^k-1$ and $m=2^{k-2}$ for a certain $k\ge 3$.
Consider the point-line geometry of $2m$-element subsets of an $n$-element set.
Maximal singular subspaces of this geometry correspond to binary simplex codes of dimension $k$.
For $k\ge 4$ the associated collinearity graph contains maximal cliques different from 
maximal singular subspaces.
We investigate maximal cliques corresponding to symmetric $(n,2m,m)$-designs.
The main results concern the case $k=4$ and give
a geometric interpretation of the five well-known
symmetric $(15,8,4)$-designs. 
\end{abstract}

\section{Introduction}
Points of the projective space ${\rm PG}(n-1,2)$ can be naturally identified with non-empty subsets of an $n$-element set.
The subgeometry ${\mathcal P}_m(n)$ formed by all $2m$-element subsets is investigated in  \cite{PPZ} 
(note that ${\mathcal P}_m(n)$ is defined only if $n\ge 3m$). 
Maximal singular subspaces of ${\mathcal P}_m(n)$ correspond to equidistant binary linear codes. 
We will suppose that $m=2^{k-2}$ and  $n=4m-1=2^{k}-1$ for a certain integer $k\ge 3$.
Then the maximal singular subspaces are related to binary simplex codes of dimension $k$. 
The geometry ${\mathcal P}_2(7)$  is a rank $3$ polar space and 
every maximal clique of the corresponding collinearity graph is a maximal singular subspace 
(a Fano plane). For $k\ge 4$ the collinearity graph contains maximal cliques which are not singular subspaces.
If such a clique consists of $n$ elements, then it  defines a symmetric $(2^{k}-1, 2^{k-1},2^{k-2})$-design
(whose points are elements of the $n$-element set and blocks are elements of the clique).

We consider so-called {\it centered} maximal $n$-element cliques of the collinearity graph of  ${\mathcal P}_m(n)$.
Every such clique is the union of $2^{k-1}-1$ lines passing through a common point and it can be obtained as a special product of two maximal $(2m-1)$-element cliques of 
the collinearity graph of ${\mathcal P}_{m/2}(2m-1)$.
Using this observation and the fact that every maximal clique of the  collinearity graph of ${\mathcal P}_2(7)$ is a Fano plane,
we show that there are precisely four types of centered maximal $15$-element cliques in the  collinearity graph of ${\mathcal P}_4(15)$
(Theorem \ref{centered4-15}). Also, we construct a maximal $15$-element cliques in the  collinearity graph of ${\mathcal P}_4(15)$
which is not centered; it is  the union of a maximal singular subspace of  ${\mathcal P}_4(15)$ (${\rm PG}(3,2)$) with a plane deleted and a plane disjoint with this subspace
(Theorem \ref{non-center}).
In this way we come to a geometric interpretation of the five well-known, pairwise non-isomorphic, symmetric $(15,8,4)$-designs. 

One way to establish our results (Theorems \ref{centered4-15} and \ref{non-center}) would be a laborious step-by-step verification based on the list of $(15,8,4)$-designs
(see, for example, \cite[Table 1.23, p. 11]{CD}).
We use more theoretical arguments showing that the classification of  centered maximal $15$-element cliques in the  collinearity graph of ${\mathcal P}_4(15)$
is equivalent to the classification of bijective transformations of a Fano plane;
this explains the emergence of these four types.
At the end, we provide a simple geometric construction of a non-centered maximal $15$-element clique.

As an application of our  approach, we simplify  the proof of Praeger-Zhou result \cite[Proposition 1.5]{PZ} 
concerning the action of automorphism groups of symmetric $(15,8,4)$-designs
(Section 6).

\section{Basics}
A {\it point-line geometry} is a pair $({\mathcal P},{\mathcal L})$,
where ${\mathcal P}$ is a non-empty set whose elements are called {\it points} and 
${\mathcal L}$ is a family formed by subsets of ${\mathcal P}$ called {\it lines}. 
Every line contains at least two points and the intersection of two distinct lines contains at most one point.
Two distinct points are  {\it collinear} if there is a line containing them.
The associated {\it collinearity graph} 
is the simple graph whose vertex set is ${\mathcal P}$ and two distinct vertices are connected by an edge if they are collinear points. 
A subset ${\mathcal X}\subset {\mathcal P}$ is called a {\it subspace} if for any two collinear points from ${\mathcal X}$
the line joining them is contained in ${\mathcal X}$. 
A subspace is  {\it singular} if any two distinct points of this subspace are collinear. 

Consider the $n$-dimensional vector space ${\mathbb F}^n$ over the two-element field ${\mathbb F}=\{0,1\}$.
The vectors
$$e_{1}=(1,0,\dots,0),\dots,e_{n}=(0,\dots,0,1)$$
form the standard basis of ${\mathbb F}^n$ 
and every non-zero vector of ${\mathbb F}^n$  is $e_I=\sum_{i\in I}e_i$,
where $I$ is a non-empty subset of the $n$-element set $[n]=\{1,\dots,n\}$, i.e.
the $i$-th coordinate of $e_I$ is $1$ if $i\in I$ and $0$ otherwise.
So, the points of the associated projective space ${\rm PG}(n-1,2)$ can be naturally  identified with
the non-empty subsets of $[n]$ and for any distinct non-empty subsets $X,Y\subset [n]$
the third point on the line containing $X,Y$ is the symmetric difference $X\triangle Y$.
Observe that the set of all $t$-element subsets of $[n]$ contains lines of ${\rm PG}(n-1,2)$
only when $t$ is even and $n\ge \frac{3}{2}t$
(see \cite{PPZ} for the details).

For a positive integer $m$ satisfying $n\ge 3m$ we
denote by  ${\mathcal P}_m(n)$ the point-line geometry whose points are all $2m$-element subsets of $[n]$
and lines are the lines of ${\rm PG}(n-1,2)$ consisting of such subsets.
Two  $2m$-element subsets $X,Y\subset [n]$ are  collinear points of  ${\mathcal P}_m(n)$ if and only if $|X\cap Y|=m$.
Maximal singular subspaces of ${\mathcal P}_m(n)$ correspond to maximal equidistant binary linear codes of length $n$ and Hamming weight $2m$.

Every permutation on the set $[n]$ (a monomial linear automorphism of ${\mathbb F}^n$) 
induces an automorphism of the geometry ${\mathcal P}_m(n)$.
All automorphisms of ${\mathcal P}_m(n)$ are determined in \cite{PPZ}. 
In some cases, for example, if the below condition \eqref{eq-simp} is satisfied, 
there are automorphisms which are not induced by permutations. 
A description of automorphisms of the collinearity graph of ${\mathcal P}_m(n)$ is an open problem,
every automorphism of the geometry is an automorphism of the collinearity graph and the converse  is not necessarily true.

\begin{rem}\label{rem-Mac}{\rm
Let ${\mathcal S}$ and ${\mathcal S}'$ be maximal singular subspaces of ${\mathcal P}_m(n)$.
Suppose that their dimension is not less than $2$.
Then every isomorphism of ${\mathcal S}$ to ${\mathcal S}'$ is induced by a linear isomorphism 
between the corresponding  equidistant binary linear codes
(the Fundamental Theorem of Projective Geometry).
The latter can be extended to a monomial linear automorphism of ${\mathbb F}^n$ by MacWilliams theorem
(see, for example, \cite[Section 7.9]{HP}).
Therefore, every isomorphism of ${\mathcal S}$ to ${\mathcal S}'$ can be extended to an automorphism of ${\mathcal P}_m(n)$
induced by a permutation.
}\end{rem}

From this moment we assume that 
\begin{equation}\label{eq-simp}
m=2^{k-2}\;\mbox{ and  }\;n=4m-1=2^{k}-1
\end{equation}
for a certain integer $k\ge 2$.
Then every maximal singular subspace of  ${\mathcal P}_m(n)$ corresponds  to a binary simplex code of dimension $k$
(by \cite{Bonis}, see also \cite[Theorem 7.9.5]{HP}, such codes  can be characterized as 
maximal equidistant binary linear codes of length $n=2^k-1$ and Hamming weight $2m=2^{k-1}$).
There is the unique binary simplex code of dimension $2$ and ${\mathcal P}_{1}(3)$ is a line formed by three points. 
The geometry ${\mathcal P}_2(7)$ is a rank $3$ polar space,  see \cite[Subsection 4.1]{KPP} and \cite[Proposition 2]{PPZ}.

By \cite{Ryser}, every maximal clique of the collinearity graph of ${\mathcal P}_m(n)$ contains at most $n$ elements.
Therefore, every maximal singular subspace of  ${\mathcal P}_m(n)$ is a maximal clique of the collinearity graph 
(it is ${\rm PG}(k-1,2)$ and, consequently, contains $2^{k}-1=n$ elements).
Since ${\mathcal P}_2(7)$ is a polar space, every maximal clique of the collinearity graph of ${\mathcal P}_2(7)$ is a maximal singular subspace
(if a point of a polar space is collinear to two distinct points on a line, then this point is collinear to all points on this line).
In the case when $k\ge 4$, there exist  maximal cliques of the collinearity graph which are not singular subspaces.
If ${\mathcal C}$  is  a maximal clique of the collinearity graph  containing $n$ elements, then 
it can be considered as a symmetric $(n,2m,m)$-design whose points are elements of $[n]$ and whose blocks are elements of ${\mathcal C}$.

Recall that a symmetric $(n,2m,m)$-design is an incidence structure formed by $n$ points and $n$ blocks.
Each block consists of $2m$ points and the intersection of two distinct blocks contains precisely $m$ points.

A well-known example of symmetric $(n,2m,m)$-design is 
the design of  points and hyperplane complements of ${\rm PG}(k-1,2)$.
Every hyperplane $H$ of ${\rm PG}(k-1,2)$ and its complement $H^c$
contain precisely  $2^{k-1}-1$ and $2^k-1-(2^{k-1}-1)=2^{k-1}=2m$ points, respectively.
The intersection of the complements of two distinct hyperplanes of ${\rm PG}(k-1,n)$ consists of $2^{k-2}=m$ points.

\begin{prop}\label{prop0}
The design of points and hyperplane complements of ${\rm PG}(k-1,2)$ is isomorphic to 
the design corresponding to a maximal singular subspace of ${\mathcal P}_m(n)$.
\end{prop}

\begin{proof}
Let $H_1$ and $H_2$ be distinct hyperplanes of ${\rm PG}(k-1,2)$.
There is a unique hyperplane $H_3$ distinct from $H_1,H_2$ and containing $H_1\cap H_2$.
We have
$$H^c_1\triangle H^c_2=H_1\triangle H_2$$
and this  set coincides with $H^c_3$
(since $H_1\cup H_2\cup H_3$ contains all points of ${\rm PG}(k-1,2)$).
So, $H^c_1,H^c_2,H^c_3$ form a line in the geometry of $2m$-element subsets of the point set of ${\rm PG}(k-1,2)$
and all complements of hyperplanes form a singular subspace of the geometry.
This subspace contains precisely $n$ elements and, consequently, it is a maximal singular subspace.
\end{proof}

\begin{rem}\label{remH}{\rm
Symmetric block designs are closely related to {\it Hadamard matrices}, 
square matrices whose entries are either $1$ or $-1$ and whose rows are mutually orthogonal.
A Hadamard matrix is {\it normalized} if its first row and column contain only $1$.
There is a one-to-one correspondence between normalized Hadamard matrices of order $4t$
and symmetric $(4t-1,2t,t)$-designs:
if in such a matrix we replace $1$ and $-1$ by $0$ and $1$ (respectively) and remove the first column and row,
then we obtain the incidence matrix of a certain symmetric $(4t-1,2t,t)$-design.
In particular, we have a one-to-one correspondence between 
maximal $n$-element cliques of the collinearity graph of ${\mathcal P}_m(n)$
and normalized Hadamard matrices of order $n+1=2^{k}$.
}\end{rem}

 \section{Centered maximal cliques}
Let ${\mathcal C}$ be a maximal clique of the collinear graph of ${\mathcal P}_{m}(n)$.
We say that $O\in {\mathcal C}$ is a {\it center point} of ${\mathcal C}$ if 
for every $C\in {\mathcal C}\setminus\{O\}$
the line joining $C$ and $O$ is in ${\mathcal C}$, in other words,
there is $C'\in {\mathcal C}$ such that $C\triangle C'=O$.
In this case, the maximal clique ${\mathcal C}$ is called {\it centered}.
A maximal clique can contain more than one center point.
For example, every maximal singular subspace of ${\mathcal P}_{m}(n)$ is a maximal clique where each point is center.

Let $O$ be a $2m$-element subset  of $[n]$, i.e a point of ${\mathcal P}_m(n)$. The complement $O^c=[n]\setminus O$ consists of 
$n-2m=2m-1$ elements.
By  \eqref{eq-simp}, we have
$$m/2=2^{k-3}\;\mbox{ and }\;2m-1=4(m/2)-1=2^{k-1}-1.$$
We take a maximal clique ${\mathcal X}$ in the collinearity graph of
the point-line geometry formed by $m$-element subsets of $O^c$
(this geometry is naturally isomorphic to ${\mathcal P}_{m/2}(2m-1)$).
Let also ${\mathcal Y}$ be a maximal clique  in the collinearity graph of
the geometry of $m$-element subsets of a certain $(2m-1)$-element subset of $O$
(the geometry  is also isomorphic to ${\mathcal P}_{m/2}(2m-1)$). 
Suppose that ${\mathcal X}$ and ${\mathcal Y}$ both contain $2m-1$ elements, i.e.
correspond to symmetric $(2m-1,m,m/2)$-designs.

We take $Y'=O\setminus Y$ for every subset $Y\subset O$
and denote by ${\mathcal Y}'$ the set consisting of all $Y'$ such that $Y\in {\mathcal Y}$.
For every bijection $\delta:{\mathcal X}\to {\mathcal Y}$ we 
denote by ${\mathcal X}\#_{\delta} {\mathcal Y}$
the subset of ${\mathcal P}_m(n)$ formed by $O$ and all $X\cup \delta(X)$, $X\cup \delta(X)'$, where $X\in {\mathcal X}$.

\begin{prop}\label{prop1}
${\mathcal X}\#_{\delta} {\mathcal Y}$ is a centered maximal $n$-element clique of the collinearity graph of ${\mathcal P}_m(n)$
and $O$ is its center point. 
\end{prop}

\begin{proof}
Every element of ${\mathcal X}\#_{\delta} {\mathcal Y}$ distinct from $O$ is $X\cup \delta(X)$ or $X\cup \delta(X)'$ for a certain $X\in {\mathcal X}$
and  it intersects $O$ in $\delta(X)$ or $\delta(X)'$, respectively.
It is easy to see that ${\mathcal X}\#_{\delta} {\mathcal Y}$ consists of $2|{\mathcal X}|+1=n$ elements
and $X\cup \delta(X)$, $X\cup \delta(X)'$, $O$ form a line for every $X\in {\mathcal X}$.

Let $X$ and $\tilde{X}$ be distinct elements of ${\mathcal X}$. 
Then $X\cup \delta(X)$ intersects $\tilde{X}\cup \delta(\tilde{X})$ and $\tilde{X}\cup \delta(\tilde{X})'$ 
in 
\begin{equation}\label{eq-intr}
(X\cap \tilde{X})\cup (\delta(X)\cap \delta(\tilde{X}))\;\mbox{ and }\; (X\cap \tilde{X})\cup (\delta(X)\cap \delta(\tilde{X})'),
\end{equation}
respectively.
Since $X,\tilde{X}$ and $\delta(X),\delta(\tilde{X})$ are distinct elements of ${\mathcal X}$ and ${\mathcal Y}$ (respectively),
we have
$$|X\cap \tilde{X}|=m/2\;\mbox{ and }\;|\delta(X)\cap \delta(\tilde{X})|=m/2.$$
The latter equality implies that $|\delta(X)\cap \delta(\tilde{X})'|=m/2$
(since $\delta(X),\delta(\tilde{X}),\delta(\tilde{X})'$ are $m$-element subsets of the $2m$-element set $O$).
Therefore, each of the subsets \eqref{eq-intr} consists of $m$ element and, consequently, 
$X\cup \delta(X)$ is collinear to $\tilde{X}\cup \delta(\tilde{X})$ and $\tilde{X}\cup \delta(\tilde{X})'$.
Similarly, we establish that $X\cup \delta(X)'$ is collinear to $\tilde{X}\cup \delta(\tilde{X})$ and $\tilde{X}\cup \delta(\tilde{X})'$.
So, ${\mathcal X}\#_{\delta} {\mathcal Y}$ is an $n$-element clique and, consequently, it is a maximal clique. 
\end{proof}

\begin{rem}{\rm 
Let $H_{\mathcal X}$ and $H_{\mathcal Y}$ be the normalized Hadamard matrices of order $2m$
corresponding to ${\mathcal X}$ and ${\mathcal Y}$, respectively.
Then the normalized Hadamard matrix of order $4m$ corresponding to ${\mathcal X}\#_{\delta} {\mathcal Y}$ is 
\setlength{\tabcolsep}{10pt}
\renewcommand{\arraystretch}{1.5}
$$
\left[
\begin{array}{c|c}
H_{\mathcal X} & \tilde{H}_{\mathcal Y}\\
\hline
H_{\mathcal X} & - \tilde{H}_{\mathcal Y}\\
\end{array}\right],
$$
where $\tilde{H}_{\mathcal Y}$ is the normalized Hadamard matrix obtained from  $H_{\mathcal Y}$ by the row permutation
corresponding to the bijection $\delta:{\mathcal X}\to {\mathcal Y}$.
}\end{rem}

\begin{prop}\label{prop2}
Let ${\mathcal C}$ be a centered maximal $n$-element clique of the collinearity graph of ${\mathcal P}_m(n)$
and let $O$ be a center point of ${\mathcal C}$.
There is a unique maximal $(2m-1)$-element clique ${\mathcal X}$ in  the collinearity graph of
the geometry formed by $m$-element subsets of $O^c$
and for every $(2m-1)$-element subset $Z\subset O$
there is a unique maximal $(2m-1)$-element clique ${\mathcal Y}$ in  the collinearity graph of
the geometry of $m$-element subsets of $Z$ such that 
${\mathcal C}={\mathcal X}\#_{\delta} {\mathcal Y}$
for a certain bijection $\delta:{\mathcal X}\to {\mathcal Y}$.
\end{prop}

\begin{proof}
If $C\in {\mathcal C}$ is distinct from $O$, then  each of the intersections 
$$X_C=C\cap O^c\;\mbox{ and }\;Y_C=C\cap O$$ 
consists of $m$ element. 
As above, we put $Y'=O\setminus Y$ for every subset $Y\subset O$.
Since $O$ is a center point,  $O\triangle C$
(the third point on the line joining $O$ and $C$) belongs to ${\mathcal C}$ and 
the equality $O\triangle C=X_C\cup Y'_C$
implies that 
$$X_{O\triangle C}=X_C\;\mbox{ and }\;Y_{O\triangle C}=Y'_C.$$
Consider the set ${\mathcal X}$ formed by all $X_C$, $C\in {\mathcal C}\setminus\{O\}$. 
Since $O\triangle C$ belongs to ${\mathcal C}$ and $X_{O\triangle C}=X_C$,
we have $$|{\mathcal X}|=(n-1)/2=2m-1.$$
Every $\tilde{C}=X_{\tilde C}\cup Y_{\tilde C}$ belonging to ${\mathcal C}\setminus\{O,C, O\triangle C\}$ intersects each of $C=X_C\cup Y_C$, $O\triangle C=X_C\cup Y'_C$
and $O$
in an $m$-element subset. 
Since
$${\tilde C}\cap C=(X_{\tilde C}\cap X_C)\cup (Y_{\tilde C}\cap Y_C),$$
$${\tilde C}\cap (O\triangle C)=(X_{\tilde C}\cap X_C)\cup (Y_{\tilde C}\cap Y'_C),$$
$${\tilde C}\cap O={\tilde C}\cap (Y_C\cup Y'_C)=(Y_{\tilde C}\cap Y_C)\cup (Y_{\tilde C}\cap Y'_C),$$
we obtain that 
$$|X_{\tilde C}\cap X_C|+|Y_{\tilde C}\cap Y_C|=|X_{\tilde C}\cap X_C|+|Y_{\tilde C}\cap Y'_C|=|Y_{\tilde C}\cap Y_C|+|Y_{\tilde C}\cap Y'_C|=m$$
and
$$|X_{\tilde C}\cap X_C|=|Y_{\tilde C}\cap Y_C|=|Y_{\tilde C}\cap Y'_C|=m/2.$$
In particular, ${\mathcal X}$ is a clique in  the collinearity graph of
the geometry of $m$-element subsets of $O^c$. This clique consists of  $2m-1$ elements and, consequently, it is maximal.

Now, we  consider  the set $\tilde{{\mathcal Y}}$ formed by all $Y_C$, $C\in {\mathcal C}\setminus\{O\}$. 
Observe that every $Y$ belongs to $\tilde{{\mathcal Y}}$ together with $Y'$
(every $C\ne O$ belongs to  ${\mathcal C}$ together with $O\triangle C$ and $Y_{O\triangle C}=Y'_C$).
If $Y_1,Y_2\in \tilde{{\mathcal Y}}$ and $Y_2\ne Y'_1$, then $|Y_1\cap Y_2|=m/2$ (it was established above).
Therefore, for any distinct $Y_1,Y_2\in \tilde{{\mathcal Y}}$ we have
$$|Y_1\cap Y_2|=m/2\;\mbox{ or }\; Y_2=Y'_1.$$
Let $Z$ be a $(2m-1)$-element subset of $O$ and let ${\mathcal Y}$ be the set of all $Y\in \tilde{{\mathcal Y}}$ contained in $Z$.
If $Y\in {\mathcal Y}$, then $Y'\not\in {\mathcal Y}$.
This means that ${\mathcal Y}$ consists of  $(n-1)/2=2m-1$ elements and is a clique in 
the collinearity graph of the geometry of $m$-element subsets of $Z$.
This clique is maximal (since it consists of $2m-1$ elements).

For every $X\in {\mathcal X}$ there is a unique $Y\in {\mathcal Y}$ such that $X\cup Y$ belongs to ${\mathcal C}$ and we take $\delta(X)=Y$.
Then ${\mathcal C}={\mathcal X}\#_{\delta} {\mathcal Y}$.
\end{proof}

\begin{rem}{\rm
Let $Z_1$ and $Z_2$ be distinct $(2m-1)$-element subsets of $O$.
Then
$${\mathcal C}={\mathcal X}\#_{\delta_1} {\mathcal Y}_1={\mathcal X}\#_{\delta_2} {\mathcal Y}_2,$$
where ${\mathcal Y}_i$ is a maximal $(2m-1)$-element clique in  the collinearity graph of
the geometry of $m$-element subsets of $Z_i$. 
Suppose that $s$ is the element of $O$ which does not belong to $Z_2$.
Then $s\in Z_1$. Replacing every $Y\in {\mathcal Y}_1$ containing $s$ by $Y'$ we obtain ${\mathcal Y}_2$.
Therefore, the multiplying some $2m$ rows (corresponding to all $Y\in {\mathcal Y}_1$ containing $s$) in the Hadamard matrix associated to ${\mathcal Y}_1$ by $-1$ gives 
the Hadamard matrix associated to ${\mathcal Y}_2$.
}\end{rem}

Let ${\mathcal P}_1,{\mathcal P}_2,{\mathcal P}'_1,{\mathcal P}'_2$  be four exemplars of ${\rm PG}(t,2)$.
We say that maps 
$$\delta:{\mathcal P}_1\to {\mathcal P}_2\;\mbox{ and }\;\delta':{\mathcal P}'_1\to {\mathcal P}'_2$$
are {\it equivalent} if there are  isomorphisms $g_i$, $i\in \{1,2\}$ of ${\mathcal P}_i$ to ${\mathcal P}'_i$
such that 
$$\delta'=g_2\delta g^{-1}_1.$$
Now, we consider centered maximal $n$-element cliques
$${\mathcal C}={\mathcal X}\#_{\delta}{\mathcal Y}\;\mbox{ and }\;{\mathcal C}'={\mathcal X}'\#_{{\delta}'}{\mathcal Y}'$$
with center points $O$ and $O'$, respectively.
As above,  ${\mathcal X},{\mathcal X}'$ are
maximal  $(2m-1)$-element cliques  in the collinearity graphs of the geometries  formed by $m$-element subsets of $O^c$ and $O'^{c}$
(respectively) and ${\mathcal Y},{\mathcal Y}'$ are 
maximal  $(2m-1)$-element cliques  in the collinearity graphs of the geometries  formed by $m$-element subsets of $(2m-1)$-element subsets $Z\subset O$ and 
$Z'\subset O'$ (respectively).

\begin{lemma}\label{iso}
If ${\mathcal X},{\mathcal Y},{\mathcal X}',{\mathcal Y}'$ are maximal singular subspaces of the corresponding geometries
and $\delta,\delta'$ are equivalent, then there is a permutation on the set $[n]$ transferring ${\mathcal C}$ to ${\mathcal C}'$.
\end{lemma}

\begin{proof}
Without loss of generality, we assume that $O=O'$ and $Z=Z'$.
By our assumption,
there are isomorphisms $g:{\mathcal X}\to {\mathcal X}'$ and $h:{\mathcal Y}\to {\mathcal Y}'$ such that 
$\delta'=h\delta g^{-1}$.
MacWillams theorem (Remark \ref{rem-Mac}) shows that  $g$ and $h$ are induced by permutations on $O^c$ and $Z$, respectively.
The second permutation can be uniquely extended  to a permutation on $O$.
The union of the above permutations is a permutation on $[n]$ sending ${\mathcal C}$ to ${\mathcal C}'$.
\end{proof}

\begin{rem}{\rm
The converse statement is not obvious. Our cliques may contain more than one center point and a permutation sending ${\mathcal C}$ to ${\mathcal C}'$
may transfer $O$ to a center point different from $O'$.
}\end{rem}

\section{Centered maximal cliques. The case $m=4$ and $n=15$}

\subsection{Classification}
In this section, we obtain  the following classification of centered maximal $15$-element cliques of the collinearity graph of ${\mathcal P}_4(15)$.

\begin{theorem}\label{centered4-15}
Let ${\mathcal C}$ be a centered maximal $15$-element clique of the collinearity graph of ${\mathcal P}_4(15)$.
Then one of the following possibilities is realized: 
\begin{enumerate}
\item[{\rm (C1)}] ${\mathcal C}$ is a maximal singular subspace of ${\mathcal P}_4(15)$ 
and every point of ${\mathcal C}$ is a center point, see Figure \ref{fig:clique1}.
\item[{\rm (C2)}] ${\mathcal C}$ is the union of three Fano planes whose intersection is a line formed by center points;
every line in ${\mathcal C}$ is contained in one of these planes, see Figure \ref{fig:clique2}.
\item[{\rm (C3)}] ${\mathcal C}$ contains a unique  Fano plane and a unique center point belonging to this plane;
every line in ${\mathcal C}$ is contained in this plane or passes through the center point, see Figure \ref{fig:clique3}. 
\item[{\rm (C4)}]  ${\mathcal C}$ contains a unique center point and every  line in ${\mathcal C}$ passes through the center point, see Figure \ref{fig:clique4}.
\end{enumerate}
For any two centered maximal $15$-element cliques of the same type there is a permutation on the set $[15]$ transferring one of these cliques to the other.
\end{theorem}

Recall that ${\mathcal P}_2(7)$ is a rank $3$ polar space and, consequently, every maximal clique of 
the collinearity graph of ${\mathcal P}_2(7)$ is a singular subspace isomorphic to a Fano plane.
By Proposition \ref{prop2}, 
every centered maximal $15$-element clique of the collinearity graph of ${\mathcal P}_4(15)$
with a center point $O$
is ${\mathcal F}_1\#_{\delta} {\mathcal F}_2$,
where ${\mathcal F}_1,{\mathcal F}_2$ are Fano planes whose points are 
$4$-element subsets of $O^c$ and $O$ (respectively) and $\delta$ is a bijection of ${\mathcal F}_1$ to ${\mathcal F}_2$.
Note that points of ${\mathcal F}_2$ are subsets of a $7$-element subset of $O$ which is not uniquely defined.
Our proof of Theorem \ref{centered4-15} is based on a classification of bijections between ${\mathcal F}_1$ and ${\mathcal F}_2$
which will be given in the next subsection.

\subsection{Bijective maps of Fano planes}
As above, we assume that ${\mathcal F}_1$ and ${\mathcal F}_2$ are Fano planes.
Recall that bijections $\delta$ and $\delta'$ of ${\mathcal F}_1$ to ${\mathcal F}_2$ are {\it equivalent} 
if there are automorphisms $g_1$ and $g_2$ of ${\mathcal F}_1$  and ${\mathcal F}_2$ (respectively)
such that
$\delta'=g_2\delta g_1$.
The {\it index} ${\rm ind}(\delta)$ of a bijection  $\delta:{\mathcal F}_1\to {\mathcal F}_2$
is the number of lines which go to lines under $\delta$. 
Therefore, $\delta$ is an isomorphism  if and only if the index of $\delta$ is $7$. 

\begin{prop}\label{Fano-map}
Two bijections between ${\mathcal F}_1$ and ${\mathcal F}_2$ are equivalent if and only if 
they are of the same index. 
There are precisely four classes of equivalence and the corresponding index values are $0,1,3,7$.
\end{prop}

Without loss of generality, we assume that
${\mathcal F}_1={\mathcal F}_2={\mathcal F}$.
The points of ${\mathcal F}$ are denoted as follows: $P_1,P_2,P_3$ are non-collinear points,
$P_{ij}$ is the third point on the line connecting $P_i$ and $P_j$, $i\ne j$  and $P_{123}$ is the seventh point
(see Fig. \ref{fig:FanoBij}).
\begin{figure}[h!]
\begin{center}
  \includegraphics[scale=0.7]{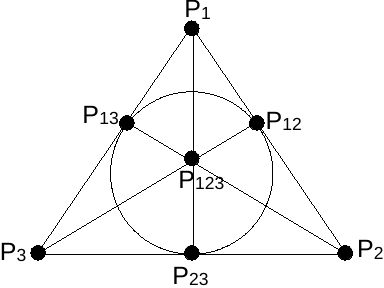}
  \caption{A labelling of the points of ${\mathcal F}$}\label{fig:FanoBij}
\end{center}
\end{figure}

Recall that a {\it simplex} in ${\mathcal F}$ is formed by four points such that any three of them are not on a line.
It is easy to see that a subset of ${\mathcal F}$ is a simplex if and only if its complement is a line.
For example, $\{P_1,P_2,P_3,P_{123}\}$ is a simplex.
For any two simplices $\{X_i\}^{4}_{i=1}$ and $\{Y_i\}^{4}_{i=1}$ there is a unique automorphism $g$
of ${\mathcal F}$ satisfying $g(X_i)=Y_i$ for every $i$.

\begin{lemma}\label{lemma-biject}
If $\delta$ is a bijective transformation of  ${\mathcal F}$ sending a certain simplex to a simplex,
then ${\rm ind}(\delta)\in \{1,3,7\}$.
A bijective transformation $\delta'$ of  ${\mathcal F}$ sending a simplex to a simplex
is equivalent to $\delta$ if and only if $\delta,\delta'$ are of the same index. 
\end{lemma}

\begin{proof}
Let $\delta$ be a bijective transformation of  ${\mathcal F}$ sending a certain simplex to a simplex.
There are automorphisms $g_1,g_2$ such that $g_2\delta g_1$
leaves fixed every point of the simplex ${\mathcal S}=\{P_1,P_2,P_3,P_{123}\}$.
Since $\delta$ and $g_2\delta g_1$ are equivalent, 
we can assume that the restriction of $\delta$ to ${\mathcal S}$ is identity.
Then $\delta$ induces a permutation on the line ${\mathcal F}\setminus {\mathcal S}$ 
formed by $P_{12},P_{13},P_{23}$
and one of the following possibilities is realized:
\begin{enumerate}
\item[$\bullet$] The restriction of $\delta$ to ${\mathcal F}\setminus {\mathcal S}$ is identity
which implies that $\delta$ is identity and, consequently, it is an automorphism. 
\item[$\bullet$] The restriction of $\delta$ to ${\mathcal F}\setminus {\mathcal S}$ is a transposition.
Then $\delta$ is of index $3$ (see Fig. \ref{fig:FanoTrans} -- the dotted lines are preserved).
\begin{figure}[h!]
\begin{center}
  \includegraphics[scale=0.7]{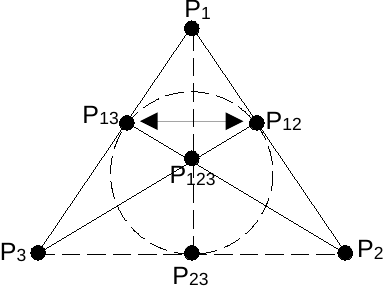} 
  \caption{The case when $\delta$ is of index $3$}\label{fig:FanoTrans}
\end{center}
\end{figure}
\item[$\bullet$] The restriction of $\delta$ to ${\mathcal F}\setminus {\mathcal S}$ is a $3$-cycle. 
Then $\delta$ is of index $1$ (see Fig. \ref{fig:FanoCycle} -- the dotted line is preserved).
\begin{figure}[h!]
\begin{center}
  \includegraphics[scale=0.7]{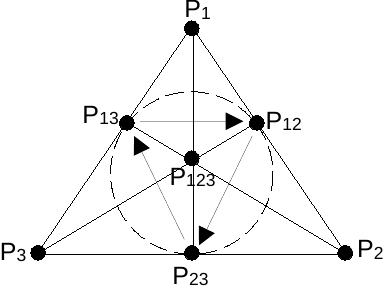}
  \caption{The case when $\delta$ is of index $1$}\label{fig:FanoCycle}
\end{center}
\end{figure}
\end{enumerate}
Suppose that $\delta'$ is a bijective transformation of  ${\mathcal F}$ which sends a  simplex to a simplex
and whose index is equal to the index of $\delta$. 
Without loss of generality, we can assume that $\delta'$ leaves fixed every point of the simplex ${\mathcal S}$.
Then the restriction of $\delta'$ to the line ${\mathcal F}\setminus {\mathcal S}$
is one of the permutations considered above.
Since ${\rm ind}(\delta)={\rm ind}(\delta')$,
there is a permutation $s$ on ${\mathcal F}\setminus {\mathcal S}$ such that 
$s^{-1}\delta s$ coincides with the restriction of $\delta'$ to ${\mathcal F}\setminus {\mathcal S}$.
We extend $s$ to an automorphism $g$ of ${\mathcal F}$ and 
obtain that $\delta'=g^{-1}\delta g$.
\end{proof}

\begin{proof}[Proof of Proposition \ref{Fano-map}]
Let $\delta$ be a bijective transformation of ${\mathcal F}$.
As above, we denote by ${\mathcal S}$ the simplex formed by $P_1,P_2,P_3,P_{123}$.
Suppose $\delta({\mathcal S})$ is not a simplex. Then $\delta({\mathcal S})$ consists of a line and a point not on this line.
Recall that any bijection between simplices can be extended to an automorphism of ${\mathcal F}$.
Also, for any lines $L_1,L_2$ in ${\mathcal F}$ and any points $P_1,P_2$ such that $P_i\not\in L_i$
every bijection of $L_1$ to $L_2$ can be extended to an automorphism of ${\mathcal F}$ transferring $P_1$ to $P_2$.
Since $\delta$ can be replaced by $g_1\delta g_2$ for any pair of automorphisms $g_1,g_2$,
we  can assume that 
$$\delta({\mathcal S})=\{P_1,P_2,P_3, P_{23}\},$$
$\delta(P_i)=P_i$ for every $i$ and $\delta(P_{123})=P_{23}$.
Then $\delta$ induces a non-trivial permutation on the set
$${\mathcal T}={\mathcal F}\setminus\{P_1,P_2,P_3\}=\{P_{12},P_{13},P_{23},P_{123}\}.$$
One of the following possibilities is realized:
\begin{enumerate}
\item[(1)] the  restriction of $\delta$ to ${\mathcal T}$ is the transposition $P_{123},P_{23}$;
\item[(2)] the  restriction of $\delta$ to ${\mathcal T}$ is the sum of the transpositions $P_{123},P_{23}$ and $P_{12},P_{13}$;
\item[(3)] the  restriction of $\delta$ to ${\mathcal T}$ is the $3$-cycle $P_{123},P_{23},P_{1i}$ with $i\in \{2,3\}$;
\item[(4)] the  restriction of $\delta$ to ${\mathcal T}$ is a $4$-cycle. 
\end{enumerate}
In the cases (1) and (2), $\delta$ sends the simplex $\{P_2,P_3,P_{12},P_{13}\}$ to itself
(it leaves fixed all points of this simplex in the first case and transposes two points in the second) and we apply Lemma \ref{lemma-biject}.
In the case (3),
$\delta$ sends the simplex $\{P_i,P_{1i},P_{23},P_{123}\}$ to itself and we use Lemma \ref{lemma-biject} again. 

Consider the case (4). A direct verification shows that the index of $\delta$ is zero.
The  restriction of $\delta$ to ${\mathcal T}$ is one of the following $4$-cycles:
$$\tau_1=P_{123},P_{23},P_{12},P_{13}\;\mbox{ or }\tau_2=P_{123},P_{23},P_{13},P_{12}.$$
Denote by $\delta_i$, $i\in \{1,2\}$ the transformation of ${\mathcal F}$ which leaves $P_1,P_2,P_3$ fixed
and whose restriction to ${\mathcal T}$ is $\tau_i$. Then $\delta$ is one of $\delta_i$.
The arguments from the beginning of our proof show that every bijective transformation of ${\mathcal F}$ with index $0$
is equivalent to one of $\delta_i$.
It is easy to see that $\delta_2=g\delta_1 g$, where $g$ is the automorphism of ${\mathcal F}$ which transposes
the pair $P_2,P_3$, the pair $P_{12},P_{13}$ and leaves fixed all remaining points.
So, $\delta_1$ and $\delta_2$ are equivalent.
\end{proof}

\subsection{Proof of Theorem \ref{centered4-15}}
Let ${\mathcal C}$ be a centered maximal $15$-element clique of the collinearity graph of ${\mathcal P}_4(15)$
and let $O$ be a center point of ${\mathcal C}$.
Then
${\mathcal C}={\mathcal F}_1\#_{\delta} {\mathcal F}_2$,
where ${\mathcal F}_1,{\mathcal F}_2$ are Fano planes whose points are 
$4$-element subsets of $O^c$ and a certain $7$-element subset of $O$ (respectively) and $\delta:{\mathcal F}_1\to{\mathcal F}_2$
is a bijection.
Denote by ${\mathcal C}_+$ and ${\mathcal C}_-$ the set of all $X\cup\delta(X)$ and, respectively,  $X\cup(O\setminus\delta(X))$ with $X\in {\mathcal F}_1$. 
Each of these sets contains precisely $7$ elements. 
For every $C\in {\mathcal C}_+$ we have $O\triangle C\in {\mathcal C}_-$.

\begin{lemma}\label{plane}
If ${\mathcal C}$ contains a certain line $L$ of ${\mathcal P}_4(15)$ which does not pass through $O$, then the plane spanned by $L$ and $O$
is contained in ${\mathcal C}$. This plane contains a line formed by points of ${\mathcal C}_+$.
\end{lemma}

\begin{proof}
The first statement is obvious.
To prove the second statement we need the following observation: there is no line of ${\mathcal P}_4(15)$ contained in ${\mathcal C}_-$.

Indeed, the points of ${\mathcal F}_2$ are $4$-element subsets of a certain $7$-element subset $Z\subset O$
and, consequently, for every $Y\in {\mathcal F}_2$ the complement $O\setminus Y$ contains the unique $i\in O\setminus Z$
which implies that every element of ${\mathcal C}_-$ contains $i$.
This gives the claim, since for any three points $A,B,C$ on a line of ${\mathcal P}_4(15)$ the intersection $A\cap B\cap C$ is empty.

Now, consider two distinct lines $L_1,L_2$ in the plane spanned by $L$ and $O$
which pass through $O$. The line $L_i$ contains a point $C_i\in {\mathcal C}_+$.
Then $O\triangle C_1, O\triangle C_2$ belong to ${\mathcal C}_-$
and the third point on the line joining them belongs to ${\mathcal C}_+$
(it was noted above that ${\mathcal C}_-$ does not contain a line).
This point is also the third point on the line joining $C_1,C_2$.
\end{proof}

Let $X_1,X_2\in {\mathcal X}$. Then the line joining $X_1\cup \delta(X_1)$ and $X_2\cup \delta(X_2)$ is contained in ${\mathcal C}_{+}$ if and only if 
$$\delta(X_1\triangle X_2)=\delta(X_1)\triangle \delta(X_2),$$
in other words, $\delta$ transfers the line joining $X_1,X_2$ in ${\mathcal F}_1$ to a certain line of ${\mathcal F}_2$.
Therefore, ${\mathcal C}_+$ contains precisely ${\rm ind}(\delta)$ lines.
By Proposition \ref{Fano-map}, ${\rm ind}(\delta)\in \{0,1,3,7\}$.

If the index is $7$, then ${\mathcal C}_+$ is a plane and ${\mathcal C}$ is the maximal singular subspace of ${\mathcal P}_4(15)$ spanned by 
this plane and $O$.

In the case when the index  is $3$, three lines of ${\mathcal F}_1$
whose images are lines of ${\mathcal F}_2$ have a common point (see Fig.\ref{fig:FanoTrans}  in the proof of Lemma \ref{lemma-biject}).
Then ${\mathcal C}_+$ is the union of three lines passing through a point which implies that 
${\mathcal C}$ is the union of three planes with a common line. 
Lemma \ref{plane} guarantees that every line contained in ${\mathcal C}$ is on one of these planes.

If the index is $1$ or $0$, then ${\mathcal C}_+$  contains a unique line or does not contain a line, respectively.
Lemma \ref{plane}  shows that the possibility (C3) or, respectively, (C4) is realized.

The second part of Theorem \ref{centered4-15} is a direct consequence of Lemma \ref{iso}.

\section{Non-centered maximal cliques. The case $m=4$ and $n=15$}
In this section, we construct a maximal $15$-element clique of the collinearity graph of ${\mathcal P}_4(15)$ which is not centered.

\begin{theorem}\label{non-center}
There is a non-centered maximal $15$-element clique in the collinearity graph of ${\mathcal P}_4(15)$
which is the union of a maximal singular subspace of ${\mathcal P}_4(15)$ with a plane deleted 
and a plane disjoint with this subspace. 
\end{theorem}

\begin{proof}
For every integer $t>0$ we define
$$[t]=\{1,\dots, t\}\;\mbox{ and }\;[-t]=\{-1,\dots, -t\}$$
and set
$$\{\pm i_1,\dots,\pm i_t\}=\{i_1,-i_1,\dots, i_t,-i_t\}$$
for any integers  $i_1,\dots,i_t >0$.
We consider the $15$-element set $[-7]\cup\{0\}\cup [7]$ instead of $[15]$.
Then
$$X_1=\{\pm 1,\pm 2, \pm 3,\pm 4\},$$
$$X_2=\{\pm 1,\pm 2,\pm 5,\pm 6\},$$
$$X_3=\{\pm 3,\pm 4,\pm 5,\pm 6\}$$
form a line of ${\mathcal P}_4(15)$
and 
$$X=\{\pm 1,\pm 3,\pm 5,\pm 7\}$$
is collinear to each $X_i$.
These  four points span a plane ${\mathcal F}$ contained in ${\mathcal P}_4(15)$ and 
the remaining three points of ${\mathcal F}$ are 
$$X\triangle X_1= \{\pm 2,\pm 4,\pm 5,\pm 7\},$$
$$X\triangle X_2= \{\pm 2,\pm 3,\pm 6, \pm 7\},$$
$$X\triangle X_3= \{\pm 1,\pm 4,\pm 6,\pm 7\}.$$
If $Y\in {\mathcal P}_4(15)$ contains $0$ and precisely one of $i,-i$ for every $i\in [7]$,
then it is collinear to all points of ${\mathcal F}$.
We take $Y=\{0\}\cup [7]$.

For any distinct $i,j\in [6]$ the point 
$$N_{ij}=\{0,i,j,7\}\cup ([-6]\setminus\{-i,-j\})$$
is  collinear to $Y$ and all points of ${\mathcal F}$.
The points $N_{ij}$ and $N_{i'j'}$ are collinear if and only if 
$$
\{i,j\}\cap\{i',j'\}=\emptyset.
$$
Indeed, 
$$N_{ij}\cap N_{i'j'}=\{0,7\}\cup (\{i,j\}\cap\{i',j'\})\cup ([-6]\setminus(\{-i,-j\}\cup \{-i',-j'\}))$$
contains precisely $4$ elements if and only if the above condition holds.

For any mutually distinct  $i,j,t\in [6]$ the point 
$$M_{ijt}=\{-7,0,i,j,t\}\cup([-6]\setminus\{-i,-j,-t\}).$$
is collinear to $Y$ and all points of ${\mathcal F}$.
The points $M_{ijt}$ and $M_{i'j't'}$ are collinear if and only if 
$$|\{i,j,t\}\cap\{i',j',t'\}|=1.$$
Indeed, $$M_{ijt}\cap M_{i'j't'}=\{-7,0\}\cup (\{i,j,t\}\cap\{i',j',t'\})\cup ([-6]\setminus(\{-i,-j,-t\}\cup \{-i',-j',-t'\}))$$
contains precisely $4$ elements if and only if the latter condition holds.
Similarly,
$N_{ij}$ and $M_{tsl}$ are collinear if and only if 
$$|\{i,j\}\cap\{t,s,l\}|=1,$$
since
$$N_{ij}\cap M_{tsl}=\{0\}\cup (\{i,j\}\cap\{t,s,l\})\cup ([-6]\setminus(\{-i,-j\}\cup\{-t,-s,-l\}))$$
contains precisely $4$ elements if and only if the latter condition is satisfied.

The points
\begin{equation}\label{eq-7points}
N_{13},\;N_{25},\;N_{46}\;\mbox{ and }\;M_{124},\;M_{156},\;M_{236},\;M_{345}
\end{equation}
are mutually collinear; see Fig. \ref{fig:non-conic-points}, where the solid lines correspond to the four points $M_{tsl}$ and the dotted lines correspond to the three points $N_{ij}$. 
\begin{figure}[h!]
\begin{center}
  \includegraphics[scale=0.7]{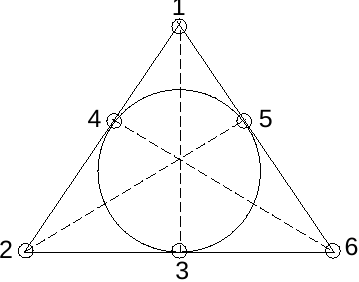}
  \caption{Schematic diagram for non-coplanar points in non-centered maximal clique}\label{fig:non-conic-points}
\end{center}
\end{figure}

For the points  \eqref{eq-7points} we have 
$$Y\triangle N_{ij}=([-6]\setminus\{-i,-j\})\cup ([6]\setminus\{i,j\}),$$
$$Y\triangle M_{ijt}=([-7]\setminus\{-i,-j,-t\})\cup ([7]\setminus\{i,j,t\}),$$
$$N_{ij}\triangle N_{i'j'}=\{\pm i,\pm j, \pm i', \pm j'\},$$
$$M_{ijt}\triangle M_{ij't'}=\{\pm j,\pm t, \pm j', \pm t'\},$$
$$N_{ij}\triangle M_{its}=\{\pm j,\pm t, \pm s, \pm 7\}.$$
More precisely,
$$Y\triangle N_{13}=N_{25}\triangle N_{46}=M_{124}\triangle M_{156}=M_{236}\triangle M_{345}=\{\pm 2,\pm4, \pm 5, \pm 6\},$$
$$Y\triangle N_{25}=N_{13}\triangle N_{46}=M_{124}\triangle M_{236}=M_{156}\triangle M_{345}=\{\pm 1, \pm3,\pm 4,\pm 6\},$$
$$Y\triangle N_{46}=N_{13}\triangle N_{25}=M_{124}\triangle M_{345}=M_{156}\triangle M_{236}=\{\pm 1,\pm2,\pm 3,\pm 5\},$$
$$Y\triangle M_{124}=N_{13}\triangle M_{156}=N_{25}\triangle M_{236}=N_{46}\triangle M_{345}=\{\pm 3,\pm5, \pm 6, \pm 7\},$$
$$Y\triangle M_{156}=N_{13}\triangle M_{124}=N_{25}\triangle M_{345}=N_{46}\triangle M_{236}=\{\pm 2,\pm3, \pm 4, \pm 7\},$$
$$Y\triangle M_{236}=N_{13}\triangle M_{345}=N_{25}\triangle M_{124}=N_{46}\triangle M_{156}=\{\pm 1,\pm4, \pm 5, \pm 7\},$$
$$Y\triangle M_{345}=N_{13}\triangle M_{236}=N_{25}\triangle M_{156}=N_{46}\triangle M_{124}=\{\pm 1,\pm2, \pm 6, \pm 7\}.$$
These seven points form a plane ${\mathcal F}'$
(the first three points are on a line and the lines joining the fourth point with the first, second and third point 
contain the fifth, sixth and seventh point, respectively).
Denote by ${\mathcal S}$  the maximal singular subspace  of ${\mathcal P}_4(15)$ spanned by ${\mathcal F}'$ and $Y$.
Then ${\mathcal S}\setminus {\mathcal F}'$ is formed by $Y$ and the points \eqref{eq-7points}.
The clique $({\mathcal S}\setminus {\mathcal F}')\cup{\mathcal F}$ contains $15$ elements and, consequently, is maximal.
This clique is not centered, since it does not contain lines passing through points of ${\mathcal S}\setminus {\mathcal F}'$.
\end{proof}

\section{Remark on the automorphism group action}
By Proposition \ref{prop0}, a clique of type (C1) (i.e. a maximal singular subspace of ${\mathcal P}_4(15)$) corresponds
to the symmetric $(15,8,4)$-design whose points are points of ${\rm PG}(3,2)$ and blocks are hyperplane complements.

\begin{prop}[Proposition 1.5 in \cite{PZ}]\label{PZ}
The design of points and hyperplane complements of ${\rm PG}(3,2)$ is the unique symmetric $(15,8,4)$-design admitting a flag-transitive, point-imprimitive subgroup of automorphisms.
\end{prop}

We prove the following.

\begin{lemma}\label{lemma-sim}
If a symmetric $(15,8,4)$-design corresponds to a clique of type {\rm (C2)-(C4)} or 
the non-centered clique from Theorem \ref{non-center}, then the automorphism group of this design
acts non-transitively on the set of blocks.
\end{lemma}

\begin{proof}
Every permutation on the set $[15]$ is an automorphism of the geometry ${\mathcal P}_4(15)$. 
So, a permutation preserving a clique of type (C2)-(C4) and sending a center point to a non-center point does not exist. 
From the same reason,  
there is no permutation preserving the non-centered clique from Theorem \ref{non-center} which sends points of
the maximal singular subspace with a plane deleted to points of the disjoint plane.
Therefore, the automorphism group of a design acts non-transitively on the set of blocks if
the blocks form a clique of one of these types.
\end{proof}

Lemma \ref{lemma-sim} together with \cite[Lemma 4.2]{PZ} (which states that the design of points and hyperplane complements of ${\rm PG}(3,2)$
admits a flag-transitive, point-imprimitive subgroup of automorphisms) imply Proposition \ref{PZ}.

\section{Figures}
All five types of maximal $15$-element cliques of the collinearity graph of ${\mathcal P}_4(15)$  together with the incidence matrices 
obtained from the corresponding normalized Hadamard matrices (see Remark~\ref{remH}) are presented on Fig.~\ref{fig:clique1} -- \ref{fig:non-centered}. Lines that are entirely contained in our cliques are pictured (except the non-centered clique with some additional dotted lines on 
Fig.~\ref{fig:non-centered}).

\begin{figure}[H]
\begin{minipage}{0.4\textwidth}
\small
\begin{tabular}{cc}
$\left(\begin{array}{c@{\ }|@{\ }ccccccccccccccc}
0&0&0&0&0&0&0&0&0&0&0&0&0&0&0&0\\
\hline
0&1&0&1&0&1&0&1&0&1&0&1&0&1&0&1\\
0&0&1&1&0&0&1&1&0&0&1&1&0&0&1&1\\
0&1&1&0&0&1&1&0&0&1&1&0&0&1&1&0\\
0&0&0&0&1&1&1&1&0&0&0&0&1&1&1&1\\
0&1&0&1&1&0&1&0&0&1&0&1&1&0&1&0\\
0&0&1&1&1&1&0&0&0&0&1&1&1&1&0&0\\
0&1&1&0&1&0&0&1&0&1&1&0&1&0&0&1\\
0&0&0&0&0&0&0&0&1&1&1&1&1&1&1&1\\
0&1&0&1&0&1&0&1&1&0&1&0&1&0&1&0\\
0&0&1&1&0&0&1&1&1&1&0&0&1&1&0&0\\
0&1&1&0&0&1&1&0&1&0&0&1&1&0&0&1\\
0&0&0&0&1&1&1&1&1&1&1&1&0&0&0&0\\
0&1&0&1&1&0&1&0&1&0&1&0&0&1&0&1\\
0&0&1&1&1&1&0&0&1&1&0&0&0&0&1&1\\
0&1&1&0&1&0&0&1&1&0&0&1&0&1&1&0
\end{array}\right)$
& 
$\begin{array}{l}
 \\1\\2\\3\\4\\5\\6\\7\\8\\9\\10\\11\\12\\13\\14\\15
\end{array}$
\end{tabular}  
\end{minipage}
\hfil
\begin{minipage}{0.4\textwidth}
  \includegraphics[width=\textwidth]{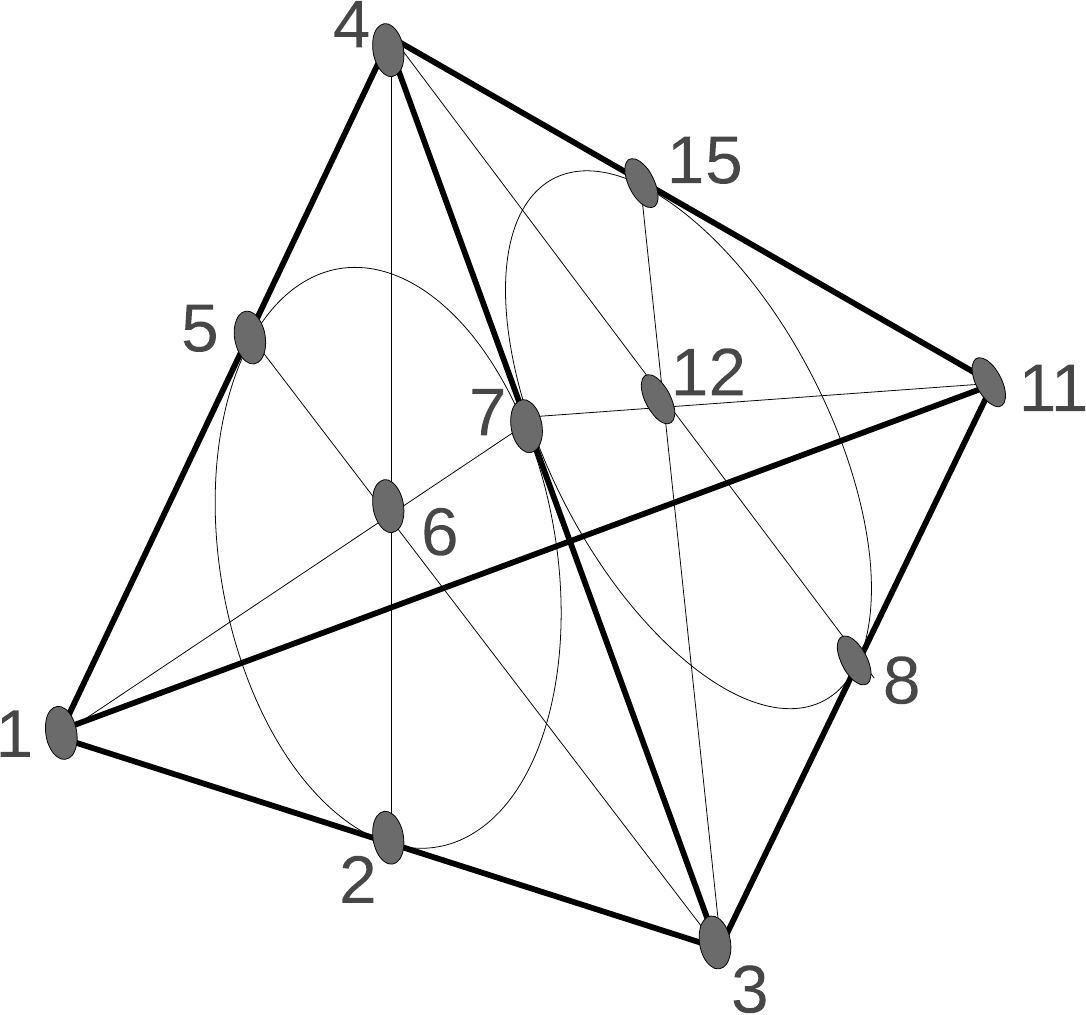}
\end{minipage}
\caption{Clique (C1) from Theorem \ref{centered4-15} and its Hadamard matrix}
\label{fig:clique1}
\end{figure}

\begin{figure}[H]
\begin{minipage}{0.4\textwidth}
\small
\begin{tabular}{cc}
$\left(\begin{array}{c@{\ }|@{\ }ccccccccccccccc}
0&0&0&0&0&0&0&0&0&0&0&0&0&0&0&0\\
\hline
0&1&0&1&0&1&0&1&0&1&0&1&0&1&0&1\\
0&0&1&1&0&0&1&1&0&0&1&1&0&0&1&1\\
0&1&1&0&0&1&1&0&0&1&1&0&0&1&1&0\\
0&0&0&0&1&1&1&1&0&0&0&0&1&1&1&1\\
0&1&0&1&1&0&1&0&0&1&0&1&1&0&1&0\\
0&0&1&1&1&1&0&0&0&0&1&1&1&1&0&0\\
0&1&1&0&1&0&0&1&0&1&1&0&1&0&0&1\\
0&0&0&0&0&0&0&0&1&1&1&1&1&1&1&1\\
0&1&0&1&0&1&1&0&1&0&1&0&1&0&0&1\\
0&0&1&1&0&0&1&1&1&1&0&0&1&1&0&0\\
0&1&1&0&0&1&0&1&1&0&0&1&1&0&1&0\\
0&0&0&0&1&1&1&1&1&1&1&1&0&0&0&0\\
0&1&0&1&1&0&0&1&1&0&1&0&0&1&1&0\\
0&0&1&1&1&1&0&0&1&1&0&0&0&0&1&1\\
0&1&1&0&1&0&1&0&1&0&0&1&0&1&0&1
\end{array}\right)$
& 
$\begin{array}{l}
 \\1\\2\\3\\4\\5\\6\\7\\8\\9\\10\\11\\12\\13\\14\\15
\end{array}$
\end{tabular}  
\end{minipage}
\hfil
\begin{minipage}{0.4\textwidth}
  \includegraphics[width=\textwidth]{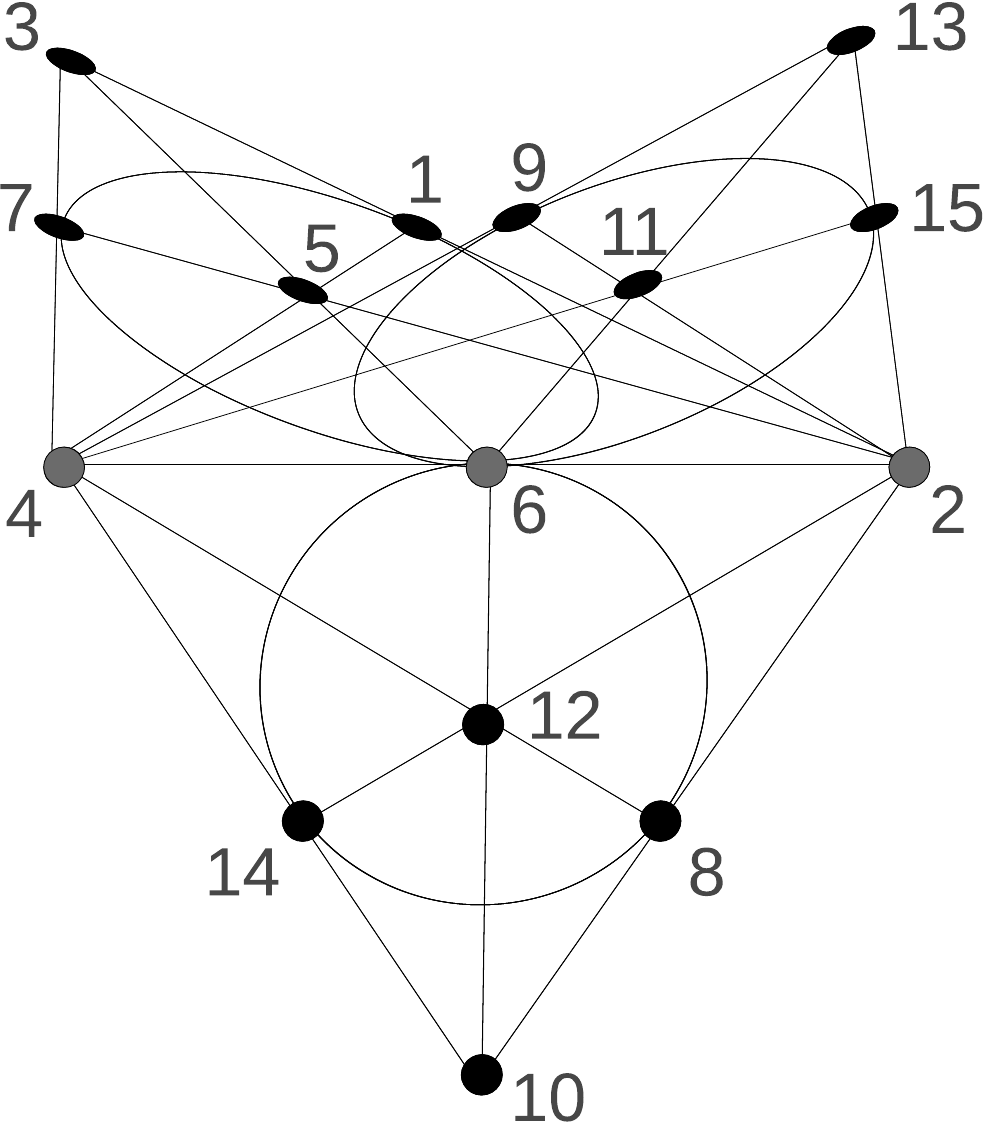}
\end{minipage}
\caption{Clique (C2) from Theorem \ref{centered4-15} and its Hadamard matrix}
\label{fig:clique2}
\end{figure}

\begin{figure}[H]
\begin{minipage}{0.4\textwidth}
\small
\begin{tabular}{cc}
$\left(\begin{array}{c@{\ }|@{\ }ccccccccccccccc}
0&0&0&0&0&0&0&0&0&0&0&0&0&0&0&0\\
\hline 
0&1&0&1&0&1&0&1&0&1&0&1&0&1&0&1\\
0&0&1&1&0&0&1&1&0&0&1&1&0&0&1&1\\
0&1&1&0&0&1&1&0&0&1&1&0&0&1&1&0\\
0&0&0&0&1&1&1&1&0&0&0&0&1&1&1&1\\
0&1&0&1&1&0&1&0&0&1&0&1&1&0&1&0\\
0&0&1&1&1&1&0&0&0&0&1&1&1&1&0&0\\
0&1&1&0&1&0&0&1&0&1&1&0&1&0&0&1\\
0&0&0&0&0&0&0&0&1&1&1&1&1&1&1&1\\
0&0&0&0&1&1&1&1&1&1&1&1&0&0&0&0\\
0&0&1&1&0&1&0&1&1&1&0&0&1&0&1&0\\
0&0&1&1&1&0&1&0&1&1&0&0&0&1&0&1\\
0&1&0&1&0&1&1&0&1&0&1&0&1&0&0&1\\
0&1&0&1&1&0&0&1&1&0&1&0&0&1&1&0\\
0&1&1&0&0&0&1&1&1&0&0&1&1&1&0&0\\
0&1&1&0&1&1&0&0&1&0&0&1&0&0&1&1\\
\end{array}\right)$
& 
$\begin{array}{l}
 \\1\\2\\3\\4\\5\\6\\7\\8\\9\\10\\11\\12\\13\\14\\15
\end{array}$
\end{tabular}  
\end{minipage}
\hfil
\begin{minipage}{0.4\textwidth}
\includegraphics[width=\textwidth]{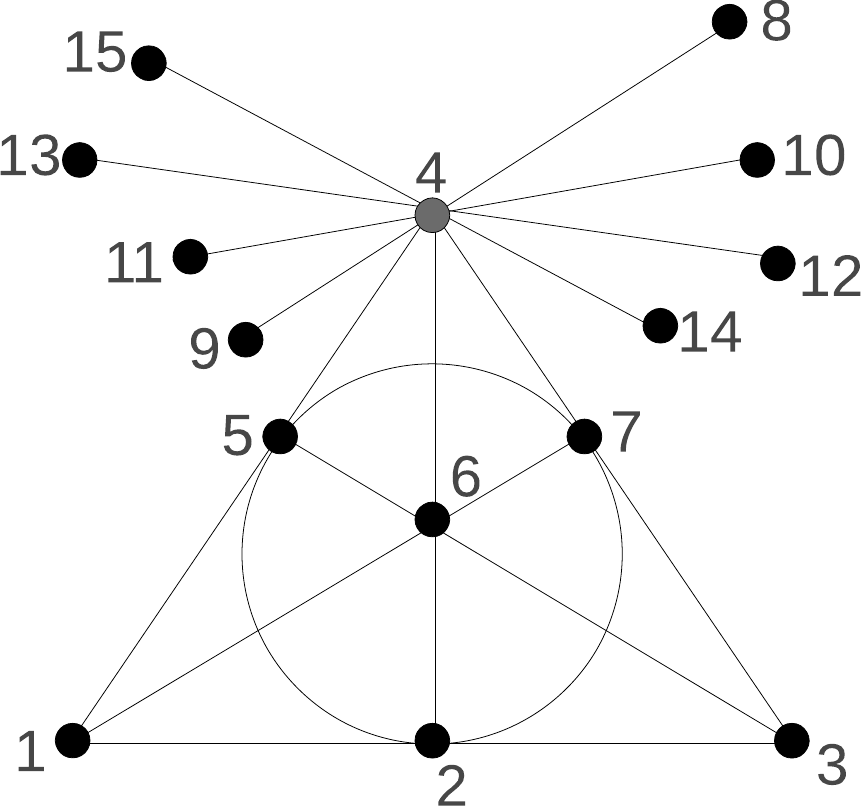}
\end{minipage}
\caption{Clique (C3) from Theorem \ref{centered4-15} and its Hadamard matrix}
\label{fig:clique3}
\end{figure}

\begin{figure}[H]
\begin{minipage}{0.4\textwidth}
\small
\begin{tabular}{cc}
$\left(\begin{array}{c@{\ }|@{\ }ccccccccccccccc}
0&0&0&0&0&0&0&0&0&0&0&0&0&0&0&0\\
\hline
0&1&0&1&0&1&0&1&0&0&0&0&1&1&1&1\\
0&0&1&1&0&0&1&1&0&0&1&1&0&0&1&1\\
0&1&1&0&0&1&1&0&0&1&0&1&0&1&0&1\\
0&0&0&0&1&1&1&1&0&0&1&1&1&1&0&0\\
0&1&0&1&1&0&1&0&0&1&1&0&0&1&1&0\\
0&0&1&1&1&1&0&0&0&1&1&0&1&0&0&1\\
0&1&1&0&1&0&0&1&0&1&0&1&1&0&1&0\\
0&0&0&0&0&0&0&0&1&1&1&1&1&1&1&1\\
0&1&0&1&0&1&0&1&1&1&1&1&0&0&0&0\\
0&0&1&1&0&0&1&1&1&1&0&0&1&1&0&0\\
0&1&1&0&0&1&1&0&1&0&1&0&1&0&1&0\\
0&0&0&0&1&1&1&1&1&1&0&0&0&0&1&1\\
0&1&0&1&1&0&1&0&1&0&0&1&1&0&0&1\\
0&0&1&1&1&1&0&0&1&0&0&1&0&1&1&0\\
0&1&1&0&1&0&0&1&1&0&1&0&0&1&0&1\\
\end{array}\right)$
& 
$\begin{array}{l}
 \\1\\2\\3\\4\\5\\6\\7\\8\\9\\10\\11\\12\\13\\14\\15
\end{array}$
\end{tabular}  
\end{minipage}
\hfil
\begin{minipage}{0.4\textwidth}
\includegraphics[width=\textwidth]{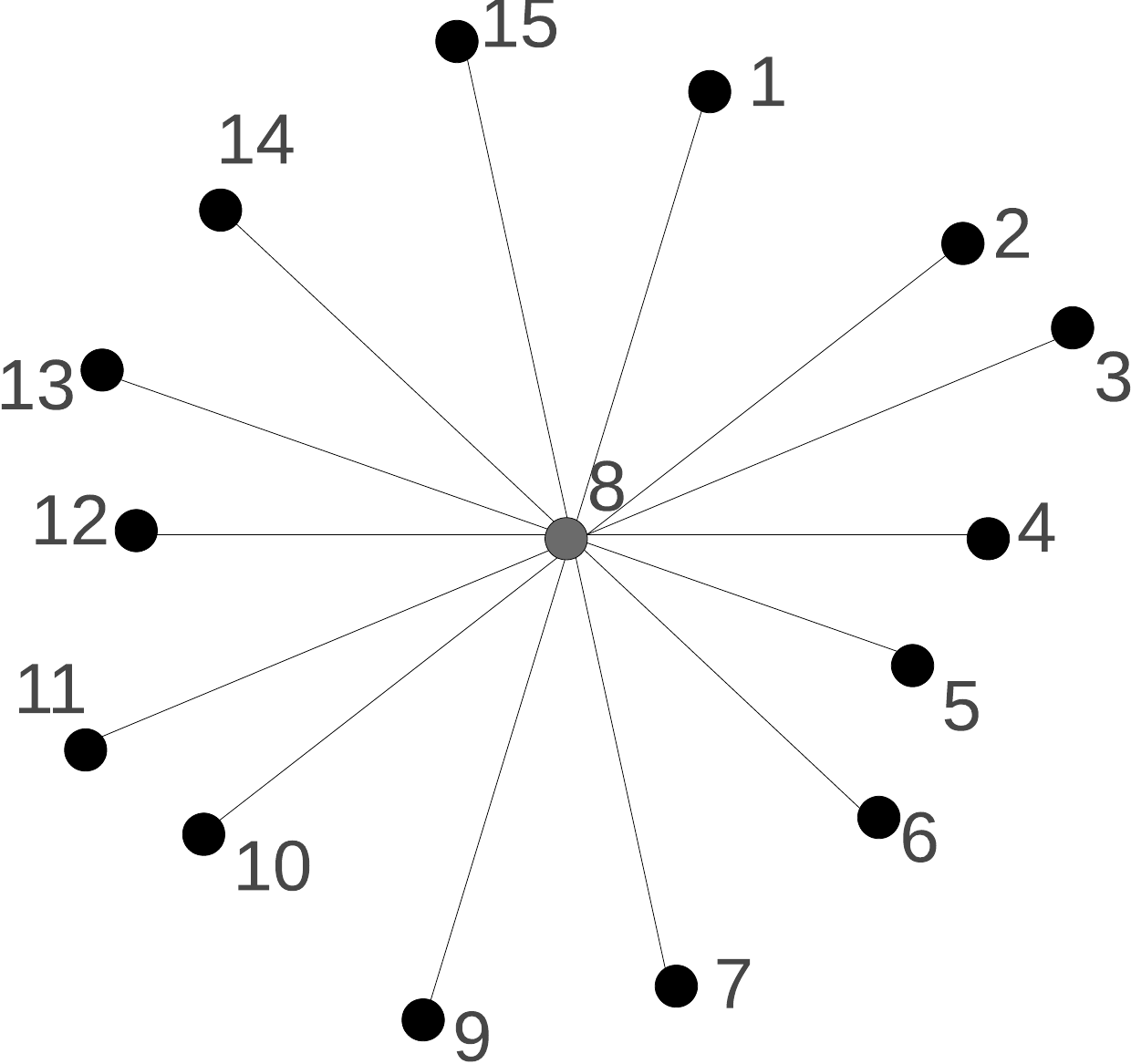}
\end{minipage}
\caption{Clique (C4) from Theorem \ref{centered4-15} and its Hadamard matrix}
\label{fig:clique4}
\end{figure}

\begin{figure}[H]
\begin{minipage}{0.4\textwidth}
\small
\begin{tabular}{cc}
$\left(\begin{array}{c@{\ }|@{\ }ccccccccccccccc}
0&0&0&0&0&0&0&0&0&0&0&0&0&0&0&0\\
\hline
0&1&0&1&0&1&0&1&0&1&0&1&0&1&0&1\\
0&0&1&1&0&0&1&1&0&0&1&1&0&0&1&1\\
0&1&1&0&0&1&1&0&0&1&1&0&0&1&1&0\\
0&0&0&0&1&1&1&1&0&0&0&0&1&1&1&1\\
0&1&0&1&1&0&1&0&0&1&0&1&1&0&1&0\\
0&0&1&1&1&1&0&0&0&0&1&1&1&1&0&0\\
0&1&1&0&1&0&0&1&0&1&1&0&1&0&0&1\\
0&0&0&0&0&0&0&0&1&1&1&1&1&1&1&1\\
0&0&0&1&0&1&1&1&1&1&1&0&1&0&0&0\\
0&0&1&0&1&1&1&0&1&1&0&1&0&0&0&1\\
0&0&1&1&1&0&0&1&1&1&0&0&0&1&1&0\\
0&1&0&0&1&0&1&1&1&0&1&1&0&1&0&0\\
0&1&0&1&1&1&0&0&1&0&1&0&0&0&1&1\\
0&1&1&0&0&1&0&1&1&0&0&1&1&0&1&0\\
0&1&1&1&0&0&1&0&1&0&0&0&1&1&0&1
\end{array}\right)$
& 
$\begin{array}{l}
 \\1\\2\\3\\4\\5\\6\\7\\8\\9\\10\\11\\12\\13\\14\\15
\end{array}$
\end{tabular}  
\end{minipage}
\hfil
\begin{minipage}{0.4\textwidth}
\includegraphics[width=0.9\textwidth]{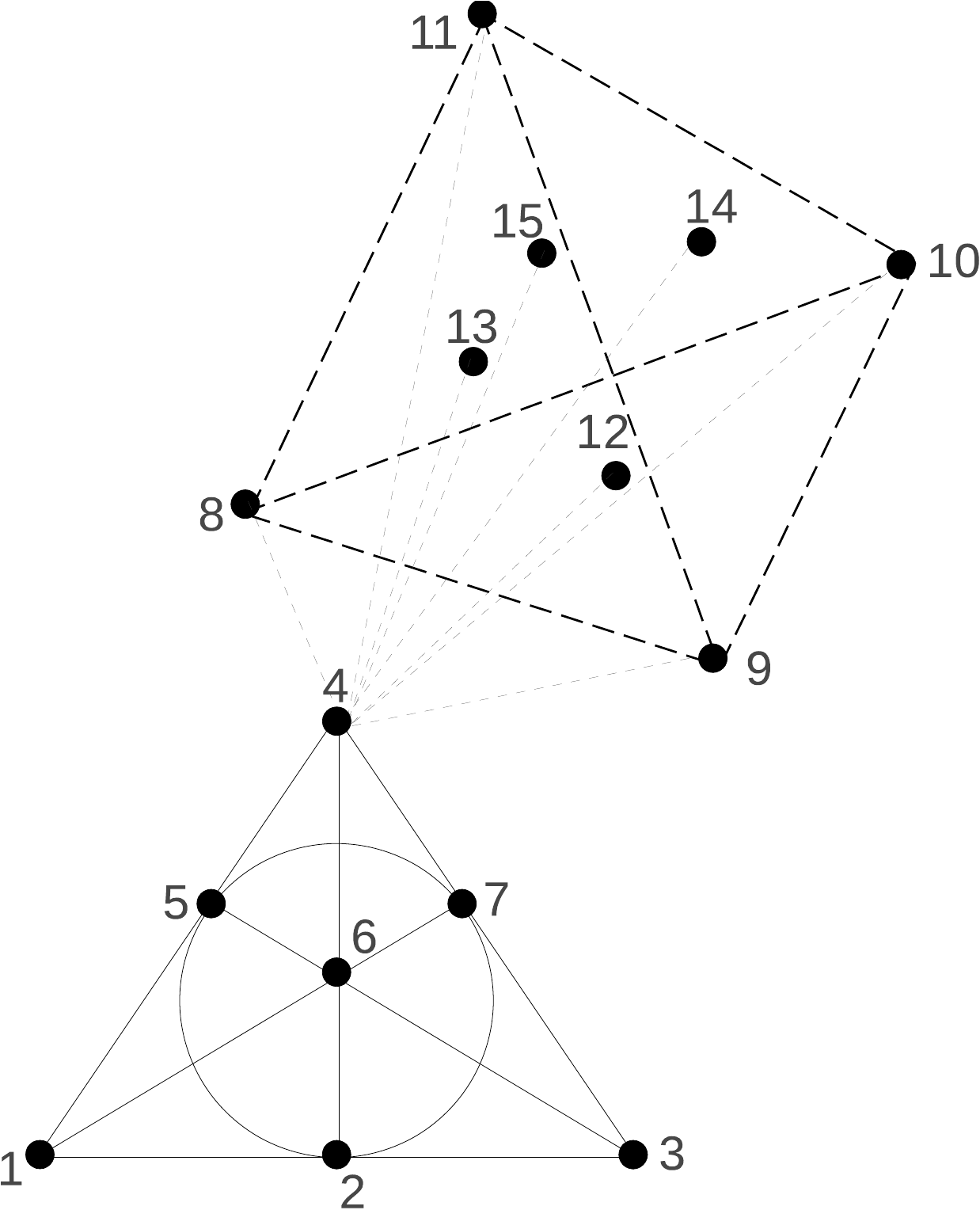}
\end{minipage}
\caption{Non-centered clique and its Hadamard matrix}
\label{fig:non-centered}
\end{figure}

\subsection*{Acknowledgment}
The authors are grateful to Alessandro Montinaro for useful information concerning \cite{PZ}.


\begin{thebibliography}{99}
\bibitem{Bonis} 
A. Bonisoli, {\it Every equidistant linear code is a sequence of dual Hamming codes},
Ars Combin. 18(1984), 181-186.

\bibitem{CD}
C. J. Colbourn, J. H. Dinitz (Eds.), {\it The CRC Handbook of Combinatorial Designs}, 
CRC Press Series on Discrete Mathematics and its Applications, CRC Press, Boca Raton 1996.

\bibitem{HP}
W. C. Huffman, V. Pless, {\it Fundamentals of Error-Correcting Codes}, 
Cambridge University Press, 2003.


\bibitem{KPP}
M. Kwiatkowski, M. Pankov, A. Pasini,
{\it The graphs of projective codes},  Finite Fields Appl. 54(2018), 15-29.

\bibitem{PPZ} M. Pankov, K. Petelczyc, M. \.Zynel, {\it Point-line geometry related to binary equidistant codes},
Journal of Combinatorial Theory, Series A (accepted).

\bibitem{PZ}
C. E. Praeger, S. Zhou, {\it Imprimitive flag-transitive symmetric designs}, 
Journal of Combinatorial Theory, Series A 113 (2006), 1381-1395.


\bibitem{Ryser}
H. J. Ryser, {\it An extension of a theorem of de Bruijn and Erd\H{o}s on combinatorial
designs}, Journal of Algebra 10 (1968), 246-261.

\end{thebibliography}
\end{document}